\documentclass[preprint,11pt]{elsarticle}
\usepackage{amsmath,eepic,amsthm}
\usepackage{amssymb}
\usepackage{amscd}
\usepackage[all]{xy}
\usepackage{color}
\usepackage{epsfig}
\usepackage{amsbsy}
\usepackage{srcltx}

\newtheorem{theorem}{Theorem}[section]
\newtheorem{lemma}[theorem]{Lemma}
\newtheorem{corollary}[theorem]{Corollary}
\newtheorem{proposition}[theorem]{Proposition}
\newtheorem{definition}[theorem]{Definition}
\newtheorem{example}[theorem]{Example}
\newtheorem{remark}[theorem]{Remark}
\newtheorem{question}[theorem]{Question}

\journal{Journal of Combinatorial Theory, Series A}

\begin{document}

\begin{frontmatter}



\title{MacWilliams-type equivalence relations}


\author{Soohak Choi}
\address{Institute of Mathematical Sciences,
Ewha Womans University, Seoul
120-750, Republic of Korea}
\ead{misb@postech.ac.kr}

\author{Jong Yoon Hyun\fnref{Hyun}}
\address{Institute of Mathematical Sciences,
Ewha Womans University, Seoul
120-750, Republic of Korea}
\ead{hyun33@postech.ac.kr}
\fntext[Hyun]{
The work of Jong Yoon Hyun was supported by the
National Research Foundation of Korea (NRF) grant funded by the Korea government (MEST) (grant \# 2011-0010328).}

\author{Dong Yeol Oh}
\address{Division of Liberal Arts, Hanbat National University, Daejeon 305-719, Republic of Korea}
\ead{dongyeol70@gmail.com}

\author{Hyun Kwang Kim\corref{Kim2}\fnref{Kim}}
\address{Department of Mathematics, POSTECH,
Pohang 790-784, Republic of Korea}
\ead{hkkim@postech.ac.kr}
\cortext[Kim2]{Corresponding author}
\fntext[Kim]{
The work of Hyun Kwang Kim was supported by Basic Research Program through the
National Research Foundation of Korea (NRF) funded by the Ministry of Education,
Science and Technology (grant \# 2010-0026473 and 2012-047640).}

\begin{abstract}
Let $\mathcal{P}$ be a poset on $[n]$, $\mathcal{I}(\mathcal{P})$
the set of order ideals of $\mathcal{P}$ and $E$ an equivalence
relation on $\mathcal{I}(\mathcal{P})$. The concepts of the dual
relation $E^*$ of an equivalence relation $E$, the $E$-weight (resp.
$E^*$-weight) distribution of a linear poset code (resp. its dual
poset code) and a MacWilliams-type equivalence relation are
introduced. We give a characterization for a MacWilliams-type
equivalence relation in terms of MacWilliams-type identities for a
linear poset code. Three kinds of equivalence relations on
$\mathcal{I}(\mathcal{P})$ which are of MacWilliams-type are found,
i.e., $(i)$ we show that every equivalence relation defined by the
automorphism of $\mathcal{P}$ is a MacWilliams-type; $(ii)$ we
provide a new characterization for poset structures when the
equivalence relation defined by the same cardinality on
$\mathcal{I}(\mathcal{P})$ becomes a MacWilliams-type; $(iii)$ we
also give necessary and sufficient conditions for poset structures
in which the equivalence relation defined by the order-isomorphism
on $\mathcal{I}(\mathcal{P})$ is a MacWilliams-type.
\end{abstract}


\end{frontmatter}


\section{Introduction}

Let $\mathbb{F}_q^n$ be the vector space of $n$-tuples over a finite field $\mathbb{F}_q$.
The space $\mathbb{F}_q^n$ endowed with the Hamming metric is called
the Hamming space. Coding theory may be considered as the study of
the Hamming space. There are several possible metrics that can be
defined on $\mathbb{F}_q^n$ \cite{BGL,G,N1,R,RT}. The ordered
Hamming space was first introduced by Niederreiter \cite{N1} to
study uniform distributions of points in the unit cube, and
developed by Rosenbloom and Tsfasman \cite{RT}; so the order
distance in the ordered Hamming space is sometimes called the
NRT-distance. The ordered Hamming space was further generalized by
Brualdi et al. \cite{BGL} to poset spaces on $\mathbb{F}_q^n$ by
assigning the coordinate positions of $\mathbb{F}_q^n$ to arbitrary
partially ordered sets.
The Hamming space and ordered Hamming space are special cases of
poset spaces given by anti-chain and the disjoint union of chains
with the same length, respectively. The poset spaces have been
extensively studied; for instances, the MacWilliams-type identity
\cite{BP,DS,K1,KO,MS,O,Skr}, perfect poset-codes \cite{HK2,KK}, the
group of (linear) isometries of the full space \cite{H,PFKH}, and
(near) MDS poset codes \cite{BP1,HK}.

One of the most fundamental results in coding theory is the MacWilliams
identity on the Hamming space which states that the Hamming weight
enumerator of a linear code is uniquely determined by that of its dual code.
The MacWilliams identity is contributed to find the maximal subsets of $\mathbb{F}_q^n$ with the given minimum Hamming distance.

There are a number of attempts to derive the MacWilliams-type
identity on $\mathbb{F}_q^n$ endowed with poset metrics; for
instances, the ordered Hamming space \cite{BP,DS,K2,MaS} and more
generally poset space \cite{K1, KO, O, Skr}. Skriganov \cite{Skr}
derived the MacWillams-type identity on chains with respect a poset
weight enumerator. Martin and Stinson \cite{MaS}, and Dougherty and
Skriganov \cite{DS} derived in different ways the MacWilliams-type
identity on ordered Hamming spaces with respect to a shape
enumerator. Kim and Oh \cite{KO} classified all poset structures
which admit the MacWilliams-type identity on poset spaces and
derived the MacWillams-type identity to such posets with respect a
poset weight enumerator.

%

The preceding discussions lead us to the following natural question:

\begin{question}
Is there a unifying way for the known results of MacWilliams-type
identities on poset spaces?
\end{question}


The paper is organized to settle Question 1.1 as follows. In Section
2, we introduce some basic concepts and notations on poset codes;
the dual relation $E^*$ of an equivalence relation $E$ (Definition
2.1), the $E$-weight (resp. $E^*$-weight) distribution of a poset
code (resp. its dual poset) and a MacWilliams-type equivalence
relation (Definition 2.7). In Section 3, we give necessary and
sufficient conditions for an equivalence relation to be a
MacWilliams-type equivalence relation (Theorem 3.3). We also derive
the connection between the $E$-weight distribution of a linear poset
code and the $E^*$-weight distribution of its dual poset code,
called the MacWilliams-type identity. It is in the matrix form whose
entries are explicitly formulated (Proposition 3.8). In Section 4,
we find equivalence relations of MacWilliams-type (Theorem 4.1),
that is, $(i)$ we show that every equivalence relation defined by
the automorphism of a poset is a MacWilliams-type; $(ii)$ we provide
a new characterization for poset structures established in \cite{KO}
with the equivalence relation defined by the cardinality on the set
of order ideals of a poset; $(iii)$ we show that every equivalence
relation defined by the order isomorphism on the set of order ideals
of a complement isomorphism poset is a MacWilliams-type and vice
versa.

\section{Preliminaries: Notations and concepts}

In this section, we review on basic definitions and notations for
poset spaces, and then define a MacWilliams-type equivalence
relation on a poset space.

Let $\mathcal{P}$ be a poset on $[n]:= \{1, 2, \ldots, n \}$ with a
partial order $\preceq$. An anti-chain is a poset whose any two
elements are incomparable. A chain is a poset whose any two elements
are comparable. A subset $I$ of $\mathcal{P}$ is an order ideal if
$a\in I$ and $b\preceq a$, then $b\in I$. Given a nonempty subset
$X$ of $[n]$, we denote $\langle X \rangle_{\mathcal{P}}$ the
smallest order ideal containing $X$.

Let $\mathcal{I}(\mathcal{P})$ denote the set of order ideals of
$\mathcal{P}$ and let $E$ be an equivalence relation on
$\mathcal{I}(\mathcal{P})$. Define the dual poset $\mathcal{P^*}$ of
$\mathcal{P}$ as follows;
$\mathcal{P}$ and $\mathcal{P^*}$ have the same underlying set and
$x \preceq y$ in $\mathcal{P}$ if and only if $y \preceq x$ in
$\mathcal{P^*}$.
It is obvious that the complement $I^c$ of $I$ in
$\mathcal{I}(\mathcal{P})$ is also an order ideal of
$\mathcal{P^*}$. Thus there is a one-to-one correspondence between
$\mathcal{I}(\mathcal{P})$ and $\mathcal{I}({\mathcal{P^*}})$. We
denote $\overline{I}$ (resp. $\overline{I^c}$) the equivalence class
of $\mathcal{I}(\mathcal{P})$ (resp. $\mathcal{I}(\mathcal{P^*})$)
containing $I$ (resp. $I^c$) with respect to $E$ (resp. $E^*$).

By $M(I)$ and $I_M$ for $I \in \mathcal{I}(\mathcal{P})$ we mean the
set of maximal elements of $I$ and nonmaximal elements of $I$,
respectively. It is obvious that $I_M$ is also an order ideal of
$\mathcal{P}$.

A permutation $\sigma$ of $\mathcal{P}$ is called an automorphism if
$\sigma$ and $\sigma^{-1}$ preserves the order relation of
$\mathcal{P}$, i.e. $x\preceq y$ if and only if $\sigma(x)\preceq
\sigma(y)$ for all $x$, $y$ in $\mathcal{P}$. It is easy to see that
the set $\textrm{Aut}(\mathcal{P})$ of all automorphisms of
$\mathcal{P}$ forms a group which is called the automorphism group
of $\mathcal{P}$.

The support supp$(x)$ and $\mathcal{P}$-weight $w_{\mathcal{P}}(x)$
of a vector $x$ in $\mathbb{F}_q^n$ are defined as
$$
\mbox{supp}(x) = \{i \mid x_i \neq 0\}\mbox{ and }w_{\mathcal{P}}(x) = |{\langle\mbox{supp}(x)\rangle}_{\mathcal{P}}|.
$$
The $\mathcal{P}$-distance between $x$ and $y$ in $\mathbb{F}_q^n$ is defined as
$$
d_{\mathcal{P}}(x, y) = w_{\mathcal{P}}(x - y).
$$
It is known \cite{BGL} that $d_{\mathcal{P}}$ is a metric on
$\mathbb{F}_q^n$, called a poset metric or a $\mathcal{P}$-metric.
If $\mathbb{F}_q^n$ is endowed with the $\mathcal{P}$-metric, then a
$($linear$)$ code of $\mathbb{F}_q^n$ is called a $($linear$)$
$\mathcal{P}$-code.\\

The following definition plays an important role for deriving the
MacWilliams-type identity.
\begin{definition}\label{definition2}
Let $\mathcal{P}$ be a poset on $[n]$ and $E$ an equivalence
relation on $\mathcal{I}(\mathcal{P})$. We say that $E^*$ is the
dual relation on $\mathcal{I}(\mathcal{P^*})$ of $E$ if it is
satisfied the following property: If $(I,J)\in E$ is defined by
property $(A)$ on $\mathcal{I}(\mathcal{P})$, then $(I^c,J^c)\in
E^*$ is also defined by property $(A)$ on
$\mathcal{I}(\mathcal{P^*})$.
\end{definition}
Definition \ref{definition2} is well-defined because $E^*$ is an
equivalence relation on $\mathcal{I}(\mathcal{P^*})$ and $E^{**}=E$. \\

We now introduce three kinds of equivalence relations on the set of
order ideals of a poset. Two of them induce naturally the dual
relation but the other does not. See Examples 2.3 and 2.5.

\begin{lemma}
Let $\mathcal{P}$ be a poset on $[n]$ and $I,J$ in $\mathcal{I}(\mathcal{P})$.\\
$(i)$ The relation $E_C$ on $\mathcal{I}(\mathcal{P})$ is defined by
the rule
\[
(I,J) \in E_C \mbox{ if and only if } |I| = |J|,
\]
Then $E_C$ is an equivalence relation on $\mathcal{I}(\mathcal{P})$
and the dual relation $E_C^*$ on $\mathcal{I}(\mathcal{P}^*)$ of
$E_C$ is
automatically determined by $|I^c|=|J^c|$. \\
$(ii)$ Let $H$ be a subgroup of $\textrm{Aut}(\mathcal{P})$. The
relation $E_{H}$ on $\mathcal{I}(\mathcal{P})$ is defined by the
rule
\begin{eqnarray*}
(I,J) \in E_{H} \textrm{ if and only if } \sigma(I) = J \textrm{ for
some } \sigma \in H.
\end{eqnarray*}
Then $E_{H}$ is an equivalence relation on
$\mathcal{I}(\mathcal{P})$ and the dual relation $E_H^*$ on
$\mathcal{I}(\mathcal{P^*})$ of $E_{H}$ is
automatically determined by $\sigma(I^c)=J^c$.\\
$(iii)$ The relation $E_S$ on $\mathcal{I}(\mathcal{P})$ is defined
by the rule
\[
(I,J) \in E_S \mbox{ if and only if } I \simeq J \mbox{ as a poset}.
\]
Then $E_S$ is an equivalence relation on $\mathcal{I}(\mathcal{P})$.
\end{lemma}
\begin{proof}
The proofs are straightforward.
\end{proof}

\begin{example}\label{example1}
Let $\mathcal{P}$ be a poset on $[5]$ with the order relation: $1
\prec 2 \prec 3$ and $4 \prec 5$.
\begin{figure}[ht]
$$\includegraphics[scale=0.5]{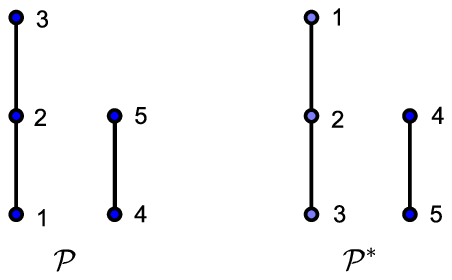}$$
\end{figure}
We see that the set $\mathcal{I}(\mathcal{P})$ becomes
\[\left\{
\emptyset, \{1\}, \{4\}, \{1,2\}, \{1,4\}, \{4,5\}, \{1,2,3\},
\{1,2,4\}, \{1,4,5\}, \{1,2,3,4\}, \{1,2,4,5\}, \mathcal{P}
\right\}.
\]
So,
$$\mathcal{I}(\mathcal{P})/E_C = \{ \overline{\emptyset}, \overline{\{1\}}, \overline{\{1,2\}}, \overline{\{1,2,3\}}, \overline{\{1,2,3,4\}}, \overline{\mathcal{P}} \},
$$
$$
\mathcal{I}(\mathcal{P^*})/E_C^* = \{ \overline{\emptyset^c},
\overline{\{1\}^c}, \overline{\{1,2\}^c}, \overline{\{1,2,3\}^c},
\overline{\{1,2,3,4\}^c}, \overline{\mathcal{P}^c} \}.
$$
Notice here that $\overline{\{1\}}=\{\{1\},\{4\}\},$
$\overline{\{1,2\}}=\{\{1,2\},\{1,4\},\{4,5\}\},$
$\overline{\{1,2,3\}}=\{\{1,2,3\},\{1,2,4\},\{1,4,5\}\},$ and
$\overline{\{1,2,3,4\}}=\{\{1,2,3,4\},\{1,2,4,5\}\}$. On the other
hand, we have
$$
\mathcal{I}(\mathcal{P})/E_S = \{ \overline{\emptyset}, \overline{\{1\}}, \overline{\{1,2\}}, \overline{\{1,4\}}, \overline{\{1,2,3\}}, \overline{\{1,2,4\}},  \overline{\{1,2,3,4\}}, \overline{\{1,2,4,5\}}, \overline{\mathcal{P}} \},
$$
$$
\mathcal{I}(\mathcal{P^*})/E_S^* \\= \{ \overline{\emptyset^c}, \overline{\{1\}^c}, \overline{\{1,2\}^c}, \overline{\{1,4\}^c}, \overline{\{1,2,3\}^c}, \overline{\{1,2,4\}^c}, \overline{\{1,2,3,4\}^c}, \overline{\{1,2,4,5\}^c}, \overline{\mathcal{P}^c} \},
$$
where $ \overline{\{1\}}=\{\{1\},\{4\}\},$
$\overline{\{1,2\}}=\{\{1,2\}, \{4,5\}\},$ and
$\overline{\{1,2,4\}}=\{\{1,2,4\},\{1,4,5\}\}$. We notice that in
this dual relation $E^*_S$ on $\mathcal{I}(\mathcal{P^*})$, every
element $(I^c, J^c)$ in $E_S^*$ is not defined by $I^c \simeq J^c$
as a poset in $\mathcal{P^*}$ because $(\{1,2\}^c,\{4,5\}^c) \in
E^*_S$, but $\{1,2\}^c$ and $\{4,5\}^c$ are not isomorphic as a
poset in $\mathcal{P^*}$. Thus the dual relation $E^*_S$ on
$\mathcal{I}(\mathcal{P^*})$ does not exist for this poset.
\end{example}
Motivated by Example 2.3, we modify the relation $E_S$ to give the
following definition.
\begin{definition}
A poset $\mathcal{P}$ is a complement isomorphism poset if the
following condition holds: for any $I$ and $J$ in
$\mathcal{I}(\mathcal{P})$,
\[
I\simeq J \mbox{ if and only if }I^c\simeq J^c.
\]
\end{definition}
An example of complement isomorphism posets is given in Figure 1.
\begin{figure}[ht]
$$\includegraphics[scale=0.5]{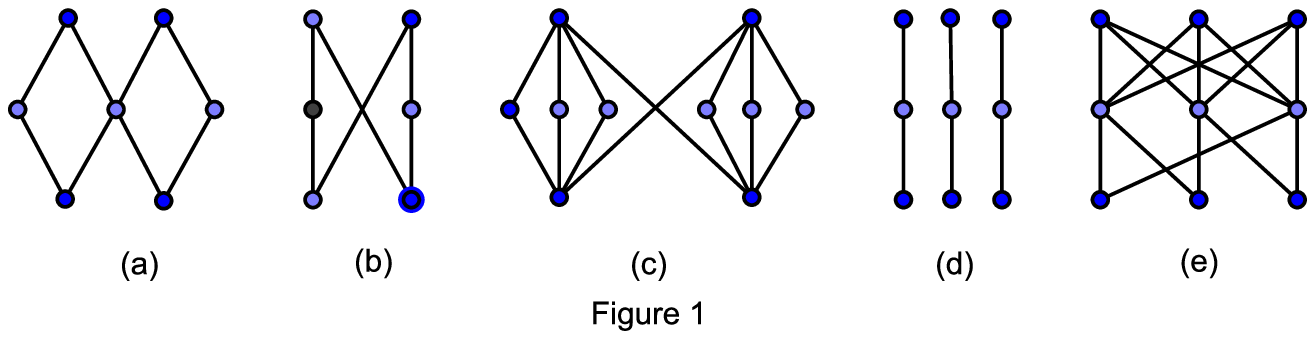}$$
\end{figure}
\begin{example}\label{example3}
Let $\mathcal{P}$ be a poset on $[4]$ with order relation: $1 \prec
3$ and $2 \prec 4$. Then $\textrm{Aut}(\mathcal{P}) = \{(1),
(12)(34)\}$. We see that
$$
\mathcal{I}(\mathcal{P})/E_{\textrm{Aut}(\mathcal{P})} = \{ \overline{\emptyset}, \overline{\{1\}}, \overline{\{1,2\}}, \overline{\{1,3\}}, \overline{\{1,2,3\}}, \overline{\mathcal{P}} \},
$$
$$
\mathcal{I}(\mathcal{P^*})/E_{\textrm{Aut}(\mathcal{P})}^* = \{ \overline{\emptyset^c}, \overline{\{1\}^c}, \overline{\{1,2\}^c}, \overline{\{1,3\}^c}, \overline{\{1,2,3\}^c}, \overline{\mathcal{P}^c}\},
$$
where $\overline{\{1\}}=\{\{1\},\{2\}\}$, $\overline{\{1,2\}}=\{\{1,2\}\}$, $\overline{\{1,3\}} = \{\{1,3\}, \{2,4\}\}$, $\overline{\{1,2,3\}}=\{\{1,2,3\},\{1,2,4\}\}$, and $\overline{\mathcal{P}} = \{\mathcal{P}\}$.
In this dual relation $E^*_{\textrm{Aut}(\mathcal{P})}$ on $\mathcal{I}(\mathcal{P^*})$, every element $(I^c, J^c)$ in $E_{\textrm{Aut}(\mathcal{P})}^*$ is also automatically determined by $\sigma(I^c) = J^c$ for some $\sigma\in Aut(\mathcal{P^*})$.
\end{example}
%

\begin{remark}
Let $\mathcal{P}$ be a poset on $[n]$. Then\\
$(i)$ $E_{Aut(\mathcal{P})} \subseteq E_S \subseteq E_C$.\\
$(ii)$ If $P$ is hierarchical (ordinal sum of anti chains), then the equalities hold.\\
$(iii)$ The equalities do not hold in general.\\
To see $(iii)$, let $\mathcal{P}$ be the poset defined in Example
$\ref{example1}$.
Put $G = Aut(\mathcal{P})$. It follows from $Aut(\mathcal{P}) = \{
1_{\mathcal{P}} \}$ that $E_G = \{(I, I) \mid I \in
\mathcal{I}(\mathcal{P}) \}$. Since $(\{1\}, \{4\}) \in E_S$ and
$(I, I) \in E_S$ for $I \in \mathcal{I}(\mathcal{P})$, we have $E_G
\subsetneq E_S$. Since $(\{1,2\}, \{1,4\}) \notin E_S$ and $|I| =
|J|$ for $(I, J) \in E_S$, we have $E_S \subsetneq E_C$.
\end{remark}

Let $I$ be an order ideal of a poset $\mathcal{P}$ on $[n]$. We
define the $I$-sphere $S_I(x)$ and the $I^c$-sphere $S_{I^c}(x)$ of
$\mathbb{F}_q^n$ centered at $x$ in $\mathbb{F}_q^n$ as follows:
$$
S_{I}(x) = \{ y \in \mathbb{F}_q^n \mid \langle \mbox{supp}(x-y) \rangle_{\mathcal{P}} = I\},
$$
$$
S_{I^c}(x) = \{ y \in \mathbb{F}_q^n \mid \langle \mbox{supp}(x-y) \rangle_{\mathcal{P^*}} = I^c \}.
$$
We also define the $\overline{I}$-sphere $S_{\overline{I},E}(x)$ and
$\overline{I^c}$-sphere $S_{\overline{I^c}, E^*}$ centered at $x$
with respect to $E$ and $E^*$ as follows:
$$
S_{\overline{I},E}(x) = \{ y \in \mathbb{F}_q^n \mid (\langle \mbox{supp}(x-y) \rangle_{\mathcal{P}}, I) \in E\},
$$
$$
S_{\overline{I^c},E^*}(x) = \{ y \in \mathbb{F}_q^n \mid (\langle \mbox{supp}(x-y) \rangle_{\mathcal{P^*}}, I^c) \in E^* \}.
$$
One can easily verify that
\begin{eqnarray*}
S_{\overline{I},E} (x) = \bigcup_{J \in \overline{I}}^{\circ} S_{J} (x)
\textrm{ and } S_{\overline{I^c},E^*} (x) = \bigcup_{J^c \in \overline{I^c}}^{\circ} S_{J^c} (x),
\end{eqnarray*}
where the union is disjoint. For the sake of simplicity, we will
write $S_{I}$ and $S_{I^c}$ (resp. $S_{\overline{I},E}$ and
$S_{\overline{I^c},E^*}$) instead of $S_{I}(\mathbf{0})$ and
$S_{I^c}(\mathbf{0})$ (resp. $S_{\overline{I},E}(\mathbf{0})$ and
$S_{\overline{I^c},E^*}(\mathbf{0})$), where $\mathbf{0}$ is the
zero vector.

Let $\mathcal{C}$ be a $\mathcal{P}$-code in $\mathbb{F}_q^n$. We
define
\begin{eqnarray*}
A_{\overline{I},E} (\mathcal{C}) := |S_{\overline{I},E} \cap
\mathcal{C} | = \sum\limits_{J \in \overline{I}} |S_{J} \cap
\mathcal{C} | \ \mbox{and} \ W(\mathcal{C}, \mathcal{P},E) := [
A_{\overline{I},E} (\mathcal{C}) ]_{\overline{I} \in
\mathcal{I}(\mathcal{P}) / E}.
\end{eqnarray*}

We call $W(\mathcal{C}, \mathcal{P},E)$ the weight distribution of
$\mathcal{C}$ with respect to $E$ (or the $E$-weight distribution of
$\mathcal{C}$). In particular, if $\mathcal{P}$ is an anti-chain on
$[n]$, then $S_{\overline{I}, E_{C}}$ $($resp. $S_{\overline{I^c},
E_{C}^*})$ is the set of vectors of $\mathbb{F}_q^n$ of Hamming
weight $|I|$ $($resp. $n - |I|)$, and the $E_C$-weight distribution
$W(\mathcal{C}, \mathcal{P},E_C)$ of $\mathcal{C}$ is just the
Hamming weight distribution of $\mathcal{C}$.

\begin{definition}\label{definition}
Let $\mathcal{P}$ be a poset on $[n]$, $E$ an equivalence relation
on $\mathcal{I}(\mathcal{P})$ and $E^*$ the dual relation on
$\mathcal{I}(\mathcal{P^*})$ of $E$. An equivalence relation $E$ on
$\mathcal{I}(\mathcal{P})$ is a MacWilliams-type if for any linear
$\mathcal{P}$-codes $\mathcal{C}_1$ and $\mathcal{C}_2$ in
$\mathbb{F}_q^n$,
\begin{eqnarray*}
W(\mathcal{C}_1, \mathcal{P},E) =  W(\mathcal{C}_2, \mathcal{P},E)
\mbox{ implies } W(\mathcal{C}_1^{\perp}, \mathcal{P^*},E^*) =
W(\mathcal{C}_2^{\perp}, \mathcal{P^*},E^*).
\end{eqnarray*}
\end{definition}

We notice that Definition \ref{definition} is well-defined because $E$ is a MacWilliams-type equivalence relation on $\mathcal{I}(\mathcal{P})$ if and only if
$E^*$ is a MacWilliams-type equivalence relation on $\mathcal{I}(\mathcal{P^*})$ using the fact that $E^{**}=E$.

\section{Equivalent conditions for a MacWilliams-type equivalence relation}

In this section, we give necessary and sufficient conditions for an
equivalence relation to be a MacWilliams-type equivalence relation.
The MacWilliams-type identites derived by our characterization are
presented in the matrix forms, say $P_E$ and $Q_{E^*}$. The entries
of $P_E$ and $Q_{E^*}$ are explicitly presented. Moreover, we prove
that $P_E$ is a uniquely determined by $Q_{E^*}$ and vice versa.

An additive character $\chi$ of $\mathbb{F}_q$ is a homomorphism
from the additive group of $\mathbb{F}_q$ into the multiplicative
group of complex numbers of absolute value one \cite{MS}. Throughout
all sections, we denote $\chi$ a nontrivial additive character of
$\mathbb{F}_q$.

\begin{lemma}\label{definition of P_I(J)}
Let $\mathcal{P}$ be a poset on $[n]$, $E$ an equivalence relation
on $\mathcal{I}(\mathcal{P})$ and $E^*$ the dual relation of $E$.
Then for any linear $\mathcal{P}$-code $\mathcal{C}$ of
$\mathbb{F}_q^n$,
\begin{eqnarray*}\label{shape weight dist3}
&(i)& \ A_{\overline{I}, E}(\mathcal{C}) = \frac{1}{|\mathcal{C}^{\perp}|}\sum\limits_{\overline{J^c} \in \mathcal{I}(\mathcal{P}^*) / E^{*} } \sum\limits_{                                                                        u \in \mathcal{C}^{\perp} \cap S_{\overline{J^c},E^{*}}} \sum\limits_{v \in S_{\overline{I}, E}} \chi(u \cdot v) \mbox{ for } \overline{I} \in \mathcal{I}(\mathcal{P})/E,\\
&(ii)& \ A_{\overline{J^c}, E^{*}}(\mathcal{C}^{\perp}) = \frac{1}{|\mathcal{C}|}\sum\limits_{\overline{I} \in \mathcal{I}(\mathcal{P}) / E } \sum\limits_{                                                                        u \in \mathcal{C} \cap S_{\overline{I},E}} \sum\limits_{v \in S_{\overline{J^c}, E^{*}}} \chi(u \cdot v) \mbox{ for } \overline{J^c} \in \mathcal{I}(\mathcal{P}^*)/E^{*}.
\end{eqnarray*}
\end{lemma}
\begin{proof}
For a linear $\mathcal{P}$-code $\mathcal{C}$ in $\mathbb{F}_q^n$, we see that
$$\mathcal{C} = \bigcup \limits_{\overline{I} \in \mathcal{I}(\mathcal{P}) / E} \mathcal{C} \cap S_{\overline{I}, E},$$
where the union is disjoint. It is well-known \cite{MS} that for any
linear $\mathcal{P}$-code $\mathcal{C}$ over $\mathbb{F}_q$,
\begin{eqnarray}\label{additive character}
\sum_{v \in \mathcal{C}} \chi(u \cdot v) =
\left\{\begin{array}{ll}
|\mathcal{C}| & \mbox{if} \ u \in \mathcal{C}^{\bot}, \\
0 & \mbox{if} \ u \not \in \mathcal{C}^{\bot}.
\end{array}\right.
\end{eqnarray}
It follows that for $\overline{J^c} \in \mathcal{I}(\mathcal{P}^*)/E^*$,
\begin{eqnarray*}
A_{\overline{J^c}, E^*}(\mathcal{C}^{\perp}) &=& \sum\limits_{v \in \mathcal{C}^{\perp} \cap S_{\overline{J^c}, E^*} } 1\nonumber\\
&=& \sum\limits_{v \in S_{\overline{J^c}, E^*} } \frac{1}{|\mathcal{C}|}\sum\limits_{u \in \mathcal{C}}\chi(u \cdot v) \  \ \ \ \ (\mbox{by }(\ref{additive character}))\nonumber\\
&=& \frac{1}{|\mathcal{C}|}\sum\limits_{u \in \mathcal{C}} \sum\limits_{v \in S_{\overline{J^c}, E^*} } \chi(u \cdot v)\nonumber\\
&=& \frac{1}{|\mathcal{C}|}\sum\limits_{\overline{I} \in \mathcal{I}(\mathcal{P})/ E} \sum\limits_{u \in \mathcal{C} \cap S_{\overline{I},E} } \sum\limits_{v \in S_{\overline{J^c}, E^*} } \chi(u \cdot v).\nonumber
\end{eqnarray*}
This proves $(ii)$. In the same way, we can obtain $(i)$.
\end{proof}

\begin{corollary}\label{definition of P_I(J)2}
Let $\mathcal{P}$ be a poset on $[n]$, $E$ an equivalence relation
on $\mathcal{I}(\mathcal{P})$ and $E^*$ the dual relation of $E$.
Then for any $1$-dimensional linear $\mathcal{P}$-code $\mathcal{C}$
of $\mathbb{F}_q^n$ generated by a nonzero vector $u$,
\begin{eqnarray*}\label{shape weight dist4}
&(i)& \ A_{\overline{I}, E}(\mathcal{C}) = \left\{\begin{array}{ll}
1 & \mbox{if} \ I = \emptyset, \\
q-1 & \mbox{if} \ u \in S_{\overline{I}, E}, \\
0 & \mbox{otherwise}.
\end{array}\right. \mbox{ for } \overline{I} \in \mathcal{I}(\mathcal{P})/E,\\
&(ii)& \ A_{\overline{J^c}, E^*}(\mathcal{C}^{\perp}) = \frac{1}{q}\left(|S_{\overline{J^c}, E^*}| + (q-1) \sum\limits_{v \in S_{\overline{J^c}, E^*} } \chi(u \cdot v)\right) \mbox{ for } \overline{J^c} \in \mathcal{I}(\mathcal{P}^*)/E^{*}.
\end{eqnarray*}
\end{corollary}
\begin{proof}
Since $\mathcal{C}$ is generated by $u$, $\mathcal{C} = \{\alpha u \mid \alpha \in \mathbb{F}_q\}$. If $u \in S_{\overline{I}, E}$, then $\alpha u \in S_{\overline{I}, E}$ for $\alpha \in \mathbb{F}_q^*$. This proves $(i)$. It follows from Lemma \ref{definition of P_I(J)} that for $\overline{J^c} \in \mathcal{I}(\mathcal{P}^*)/E^{*}$,
\begin{eqnarray*}
A_{\overline{J^c}, E^{*}}(\mathcal{C}^{\perp}) &=& \frac{1}{|\mathcal{C}|}\sum\limits_{\overline{I} \in \mathcal{I}(\mathcal{P}) / E } \sum\limits_{w \in \mathcal{C} \cap S_{\overline{I},E} } \sum\limits_{v \in S_{\overline{J^c}, E^{*}}} \chi(w \cdot v) \nonumber\\
&=& \frac{1}{q} \sum\limits_{\alpha \in \mathbb{F}_q} \sum\limits_{v \in S_{\overline{J^c}, E^*}} \chi((\alpha u) \cdot v) \nonumber\\
&=& \frac{1}{q}\left(\sum\limits_{v \in S_{\overline{J^c}, E^{*}}} \chi(0 \cdot v) + \sum\limits_{\alpha \in \mathbb{F}^*_q} \sum\limits_{v \in S_{\overline{J^c}, E^{*}}} \chi((\alpha u) \cdot v) \right)\nonumber\\
&=& \frac{1}{q}\left(|S_{\overline{J^c}, E^{*}}|  + \sum\limits_{\alpha \in \mathbb{F}^*_q} \sum\limits_{v \in S_{\overline{J^c},E^{*}}} \chi(u \cdot (\alpha v)) \right).\nonumber
\end{eqnarray*}
Since $S_{\overline{J^c}, E^*} = \{\alpha v\mid v  \in S_{\overline{J^c}, E^*}\}$ for $\alpha \in \mathbb{F}_q^*$, we have
\begin{eqnarray*}
A_{\overline{J^c}, E^{*}}(\mathcal{C}^{\perp}) &=& \frac{1}{q}\left(|S_{\overline{J^c}, E^{*}}| + (q-1) \sum\limits_{v \in S_{\overline{J^c}, E^{*}}} \chi(u \cdot v)\right).\nonumber
\end{eqnarray*}
This proves $(ii)$.
\end{proof}

We are ready to state equivalent conditions for the MacWilliams-type
equivalence relation.

\begin{theorem}\label{S.M.I.2}
Let $\mathcal{P}$ be a poset on $[n]$, $E$ an equivalence relation
on $\mathcal{I}(\mathcal{P})$
and $E^*$ the dual relation of $E$. The following statements are equivalent.\\
$(i)$ $E$ is a MacWilliams-type equivalence relation on $\mathcal{I}(\mathcal{P})$.\\
$(ii)$ For $\overline{I} \in \mathcal{I}(\mathcal{P}) / E$ and $\overline{J^c} \in \mathcal{I}(\mathcal{P^*}) / E^*$, we have\\
\mbox{ }  \ \ \ $(a)$ If $u$ and $u'$ are in $S_{\overline{I},E}$, then $\sum\limits_{v \in S_{\overline{J^c}, E^*} } \chi(u \cdot v) = \sum\limits_{v \in S_{\overline{J^c}, E^*} } \chi(u' \cdot v)$.\\
\mbox{ } \ \ \ $(b)$ If $v$ and $v'$ are in $S_{\overline{J^c}, E^*}$, then $\sum\limits_{u \in S_{\overline{I},E} } \chi(u \cdot v) = \sum\limits_{u \in S_{\overline{I},E} } \chi(u \cdot v')$.\\
$(iii)$ There are matrices $Q_{E^{*}}$ and $P_E$ over $\mathbb{F}_q$ such that for any linear $\mathcal{P}$-code $\mathcal{C}$ in $\mathbb{F}_q^n$, we have\\
\mbox{ }  \ \ \ $(a)$ $W(\mathcal{C}^{\perp}, \mathcal{P^*},E^*) = \frac{1}{|\mathcal{C}|} W(\mathcal{C}, \mathcal{P},E) Q_{E^{*}}$.\\
\mbox{ }  \ \ \ $(b)$ $W(\mathcal{C}, \mathcal{P},E) =  \frac{1}{|\mathcal{C}^{\perp}|} W(\mathcal{C}^{\perp}, \mathcal{P^*},E^*) P_E$.
\end{theorem}
\begin{proof}
$(i)\Rightarrow(ii)$ Suppose an equivalence relation $E$ on
$\mathcal{I}(\mathcal{P})$ doesn't admit either $(a)$ in $(ii)$  or
$(b)$ in $(ii)$. Without loss of generality, we assume that there
are $u$ and $u'$ in $S_{\overline{I},E}$ such that $\sum_{v \in
S_{\overline{J^c}, E^*} } \chi(u \cdot v) \neq \sum_{v \in
S_{\overline{J^c}, E^*} } \chi(u' \cdot v)$. Let $\mathcal{C}_1$ and
$\mathcal{C}_2$ be 1-dimensional codes of $\mathbb{F}_q^n$ generated
by $u$ and $u'$, respectively. It follows from Corollary
\ref{definition of P_I(J)2} that $W(\mathcal{C}_1, \mathcal{P},E) =
W(\mathcal{C}_2, \mathcal{P},E)$ and $W(\mathcal{C}^{\perp}_1,
\mathcal{P^*}, E^*) \neq W(\mathcal{C}^{\perp}_2, \mathcal{P^*},
E^*)$. So $E$ is not a MacWilliams-type equivalence relation on
$\mathcal{I}(\mathcal{P})$.

$(ii)\Rightarrow(i)$ Suppose an equivalence relation $E$ on
$\mathcal{I}(\mathcal{P})$ admits $(a)$ and $(b)$ in $(ii)$. We
claim that for linear $\mathcal{P}$-codes $\mathcal{C}_1$ and
$\mathcal{C}_2$ in $\mathbb{F}_q^n$,
\begin{eqnarray*}
W(\mathcal{C}_1, \mathcal{P},E) =  W(\mathcal{C}_2, \mathcal{P},E) \mbox{ if and only if } W(\mathcal{C}_1^{\perp}, \mathcal{P^*},E^*) = W(\mathcal{C}_2^{\perp}, \mathcal{P^*},E^*).
\end{eqnarray*}
Assume that $W(\mathcal{C}_1, \mathcal{P},E) =  W(\mathcal{C}_2,
\mathcal{P},E)$. Since the equivalence relation $E$ admits $(a)$ in
$(ii)$, the summation $\sum_{v \in S_{\overline{J^c},E*} } \chi(u
\cdot v)$ is a constant for any $u \in S_{\overline{I}, E}$. Put
$p_{_{\overline{J^c}, \overline{I}}} = \sum_{v \in
S_{\overline{J^c}, E^*} } \chi(u \cdot v)$ for $u \in
S_{\overline{I},E}$. If follows from Lemma \ref{definition of
P_I(J)} that for $j=1,2,$
\begin{eqnarray}
A_{\overline{J^c}, E^*}(\mathcal{C}^{\perp}_j) &=&
\frac{1}{|\mathcal{C}_j|}\sum\limits_{\overline{I} \in
\mathcal{I}(\mathcal{P}) / E^* } \sum\limits_{w\in \mathcal{C}_j
\cap S_{\overline{I},E}  }
\sum\limits_{v \in S_{\overline{J^c},E^*}} \chi(w \cdot v) \nonumber\\
&=& \frac{1}{|\mathcal{C}_j|}\sum\limits_{\overline{I} \in
\mathcal{I}(\mathcal{P}) / E^* } A_{\overline{I}, E}(\mathcal{C}_j)
p_{_{\overline{J^c}, \overline{I}}},\label{shape weight dist4}
\end{eqnarray}
which implies that $W(\mathcal{C}_1^{\perp}, \mathcal{P^*},E^*) = W(\mathcal{C}_2^{\perp}, \mathcal{P^*},E^*)$.

By the same argument as above, we can prove the other direction. Thus $E$ is a MacWilliams-type equivalence relation on $\mathcal{I}(\mathcal{P})$.

$(ii)\Rightarrow(iii)$ Suppose an equivalence relation $E$ on
$\mathcal{I}(\mathcal{P})$ admits $(a)$ and $(b)$ in $(ii)$. For
$\overline{I} \in \mathcal{I}(\mathcal{P}) / E$ and $\overline{J^c}
\in \mathcal{I}(\mathcal{P^*}) / E^*$, the summations $\sum_{v \in
S_{\overline{J^c},E*} } \chi(u \cdot v)$ and $\sum_{u \in
S_{\overline{I},E} } \chi(u \cdot v)$ are constants for $u \in
S_{\overline{I}, E}$ and $v \in S_{\overline{J^c}, E^*}$. Define the
matrix $P_E$ and $Q_{E^*}$ as follows:
\begin{align}
P_E = [p_{_{\overline{J^c}, \overline{I}}}] \mbox{ and }
Q_{E^*} = [q_{_{\overline{I}, \overline{J^c}}}],
\end{align}
where $p_{_{\overline{J^c}, \overline{I}}} = \sum_{v \in S_{\overline{J^c}, E^*} } \chi(u \cdot v)$ for $u \in S_{\overline{I},E}$ and $q_{_{\overline{I}, \overline{J^c}}} = \sum_{u \in S_{\overline{I},E} } \chi(u \cdot v)$ for $v \in S_{\overline{J^c}, E^*}$. Here $P_E$ is an
$|\mathcal{I}(\mathcal{P^*})/E^*|\times|\mathcal{I}(\mathcal{P})/E|$
matrix with rows and columns labelled by the elements of
$\mathcal{I}(\mathcal{P^*})/E^*$ and of
$\mathcal{I}(\mathcal{P})/E$, respectively, and $Q_E$ is an
$|\mathcal{I}(\mathcal{P})/E|\times|\mathcal{I}(\mathcal{P^*})/E^*|$
matrix with rows and columns labelled by the elements of
$\mathcal{I}(\mathcal{P})/E$ and of
$\mathcal{I}(\mathcal{P^*})/E^*$, respectively.
If follows from $(\ref{shape weight dist4})$ that $W(\mathcal{C}^{\perp}, \mathcal{P^*},E^*) = \frac{1}{|\mathcal{C}|} W(\mathcal{C}, \mathcal{P},E) Q_{E^{*}}$ for any linear $\mathcal{P}$-code $\mathcal{C}$ in $\mathbb{F}_q^n$. In the same way, we can obtain $W(\mathcal{C}, \mathcal{P},E) =  \frac{1}{|\mathcal{C}^{\perp}|} W(\mathcal{C}^{\perp}, \mathcal{P^*},E^*) P_E$ for any linear $\mathcal{P}$-code $\mathcal{C}$ in $\mathbb{F}_q^n$.

$(iii)\Rightarrow(i)$ Suppose an equivalence relation $E$ on
$\mathcal{I}(\mathcal{P})$ admits $(a)$ and $(b)$ in $(iii)$. We
claim that for linear $\mathcal{P}$-codes $\mathcal{C}_1$ and
$\mathcal{C}_2$ in $\mathbb{F}_q^n$,
\begin{eqnarray*}
W(\mathcal{C}_1, \mathcal{P},E) =  W(\mathcal{C}_2, \mathcal{P},E) \mbox{ if and only if } W(\mathcal{C}_1^{\perp}, \mathcal{P^*},E^*) = W(\mathcal{C}_2^{\perp}, \mathcal{P^*},E^*).
\end{eqnarray*}
Assume that $W(\mathcal{C}_1, \mathcal{P},E) =  W(\mathcal{C}_2,
\mathcal{P},E)$. Since the equivalence relation $E$ admits $(a)$ in
$(iii)$, we have
\begin{eqnarray*}
W(\mathcal{C}_1^{\perp}, \mathcal{P^*},E^*) = \frac{1}{|\mathcal{C}_1|} W(\mathcal{C}_1, \mathcal{P},E) Q_{E^{*}} = \frac{1}{|\mathcal{C}_2|} W(\mathcal{C}_2, \mathcal{P},E) Q_{E^{*}} = W(\mathcal{C}_2^{\perp}, \mathcal{P^*},E^*).
\end{eqnarray*}

By the same argument as above, we can prove the other direction. Thus $E$ is a MacWilliams-type equivalence relation on $\mathcal{I}(\mathcal{P})$.
\end{proof}

\begin{definition}
Let $\mathcal{P}$ be a poset on $[n]$ and $E$ a MacWilliams-type
equivalence relation on $\mathcal{I}(\mathcal{P})$. We call the
matrix $P_E$ defined in $(3)$ the $P$-matrix with respect to $E$ and
call the matrix $Q_{E^*}$ defined in $(3)$ the $Q$-matrix with
respect to $E^*$.
\end{definition}
From now on, we try to find out formulae for the entries of $P_E$
and $Q_{E^*}$.
\begin{lemma}\label{sphere}
Let $\mathcal{P}$ be a poset on $[n]$. For $I \in \mathcal{I}(\mathcal{P})$, we have
$$
S_{I} =
\left\{ (v_1,v_2,\ldots,v_n) \in \mathbb{F}_q^n \mid v_i \in \left[ \begin{array}{ll} \mathbb{F}_q^* &\mbox{ if } i \in M(I),\\
\mathbb{F}_q &\mbox{ if } i \in I_M,\\
\{0\}&
\mbox{ if } i \in I^c. \end{array}\right.
\right\}.
$$
\end{lemma}
\begin{proof}
From the definition of the $I$-sphere $S_I$, we have
\begin{eqnarray*}
S_{I} = \{ v \in \mathbb{F}_q^n \mid \langle \mbox{supp}(v) \rangle_{\mathcal{P}} = I\}
\end{eqnarray*}
Since $\langle \mbox{supp}(v) \rangle_{\mathcal{P}} = I$ if and only if
$M(I) \subseteq \mbox{supp}(v) \subseteq I$, we have the result.
\end{proof}

\begin{lemma}\label{support of u}
Let $\mathcal{P}$ be a poset on $[n]$. For $I$ and $J$ in $\mathcal{I}(\mathcal{P})$, the following statements are equivalent.\\
$(i)$ $\mbox{supp}(u) \cap (J^c)_M = \emptyset$ for $u \in S_I$.\\
$(ii)$ $M(I) \cap (J^c)_M = \emptyset$.\\
$(iii)$ $I \cap (J^c)_M = \emptyset$.\\
$(iv)$ $I_M \cap J^c = \emptyset$.
\end{lemma}
\begin{proof}
$(i) \Rightarrow (ii)$ For $u \in S_I$, $M(I) \subseteq \mbox{supp}(u)$.
It follows that $M(I) \cap (J^c)_M \subseteq \mbox{supp}(u) \cap (J^c)_M$. Hence $(i)$ implies $(ii)$.\\
$(ii) \Rightarrow (iii)$ Note that $\{z \in M(I) \mid x \preceq z \mbox{ in } \mathcal{P}\} \neq \emptyset$
for $x \in I_M$ and $(J^c)_M$ is an order ideal of $\mathcal{P}^*$.
If $x \in I_M \cap (J^c)_M$, then $y \in M(I) \cap (J^c)_M$ for $y \in \{z \in M(I) \mid x \preceq z \mbox{ in } \mathcal{P}\}$.
Hence $(ii)$ implies $(iii)$.\\
$(iii) \Rightarrow (iv)$ If $x \in I_M \cap J^c$, then $y \in I \cap
(J^c)_M$
for $y \in \{z \in M(I) \mid x \preceq z \mbox{ in } \mathcal{P}\}$. Hence $(iii)$ implies $(iv)$.\\
$(iv) \Rightarrow (i)$ Note that $\{z \in M(J^c) \mid x \preceq z
\mbox{ in } \mathcal{P}^*\} \neq \emptyset$ for $x \in (J^c)_M$. If
$x \in \mbox{supp}(u) \cap (J^c)_M$, then $y \in I_M \cap J^c$ for
$y \in \{z \in M(J^c) \mid x \preceq z \mbox{ in } \mathcal{P}^*\}$.
Hence $(iv)$ implies $(i)$.
\end{proof}

We evaluate the sum of characters on the sphere of an order ideal.

\begin{lemma}\label{calculation of P_I(J)}
Let $\mathcal{P}$ be a poset on $[n]$. For $I, J \in \mathcal{I}(\mathcal{P})$ and $u\in S_I$, we have
\begin{eqnarray*}
\sum\limits_{v \in S_{J^c}}
                                                \chi(u \cdot v)=
\left\{\begin{array}{ll}    (-1)^{|I \cap J^c|}(q-1)^{|M(J^c)|-|I \cap J^c|}q^{|(J^c)_M|} & \mbox{if }I_M \cap J^c = \emptyset,\\
0 &\mbox{if }I_M \cap J^c \neq \emptyset. \end{array}\right.
\end{eqnarray*}
\end{lemma}
\begin{proof}
It follows from Lemma \ref{sphere} that
\begin{eqnarray}
\sum\limits_{v \in S_{J^c}}
                                                \chi(u \cdot v) &=& \sum\limits_{v \in S_{J^c}}
                                                \prod\limits_{i=1}\limits^{n}\chi(u_i  v_i)\nonumber\\
&=& \sum\limits_{v \in S_{J^c}}
                                                \prod\limits_{i \in M(J^c)}
                                                \chi(u_i  v_i)
                                                \prod\limits_{i \in J_M^c}
                                                \chi(u_i  v_i)
                                                \prod\limits_{i \in J}
                                                \chi(u_i  v_i)
                                                \nonumber\\
&=& \prod\limits_{i \in M(J^c)} \sum\limits_{\alpha \in F^*_q}\chi(u_i \alpha) \prod\limits_{i \in J^c_M} \sum\limits_{\alpha \in F_q}\chi(u_i  \alpha)\prod\limits_{i \in J} \chi(0)  \nonumber
\end{eqnarray}
Since
$
\sum_{\beta \in \mathbb{F}_q} \chi(\alpha \beta) =
\left\{\begin{array}{ll}
q & \mbox{if} \ \alpha = 0,\\
0 & \mbox{if} \ \alpha \neq 0,
\end{array}\right.
$
we have
\begin{eqnarray*}
& &\sum\limits_{v \in S_{J^c}} \chi(u \cdot v)\nonumber\\
&=& (-1)^{|supp(u) \cap M(J^c)|}(q-1)^{|supp(u)^c \cap M(J^c)|} 0^{|supp(u) \cap (J^c)_M|} q^{|supp(u)^c \cap (J^c)_M|}\nonumber\\
&=& \left\{\begin{array}{ll}    (-1)^{|supp(u) \cap M(J^c)|}(q-1)^{|supp(u)^c \cap M(J^c)|} q^{|(J^c)_M|} & \mbox{if }supp(u) \cap (J^c)_M = \emptyset,\\
0 &\mbox{if }supp(u) \cap (J^c)_M \neq \emptyset. \end{array}\right.\nonumber
\end{eqnarray*}
The result follows from Lemma \ref{support of u}.
\end{proof}
In the following proposition, the entries of $P_E$ and $Q_{E^*}$ are
explicitly described.
\begin{proposition}\label{calculation of P_I(J)2}
Let $\mathcal{P}$ be a poset on $[n]$, $E$ an equivalence relation
on $\mathcal{I}(\mathcal{P})$ and $E^*$ the dual relation of $E$.
Then the entries of $P_E$ and $Q_{E^*}$ are presented as follows:

For $I, J \in \mathcal{I}(\mathcal{P})$ and $u\in S_I$, we have
\begin{eqnarray*}
(i) \ p_{\overline{J^c},\overline{I}}= (q-1)^{|M(J^c)|}q^{|(J^c)_M|}
\sum\limits_{K^c \in \overline{J^c}, I_M \cap K^c = \emptyset}
\left(\frac{-1}{q-1}\right)^{|I \cap K^c|} \mbox{ for }u\in
S_{\overline{I},E},
\end{eqnarray*}
\begin{eqnarray*}
(ii) \ q_{\overline{I},\overline{J^c}}=(q-1)^{|M(I)|}q^{|I_M|}
\sum\limits_{K \in \overline{I}, (J^c)_M \cap K = \emptyset}
\left(\frac{-1}{q-1}\right)^{|J^c \cap K|} \mbox{ for } v\in
S_{\overline{J^c},E^*}.
\end{eqnarray*}
\end{proposition}
\begin{proof}
It follows from Lemma \ref{calculation of P_I(J)} that
\begin{eqnarray}
\sum\limits_{v \in S_{\overline{J^c}, E^*}}
                                                \chi(u \cdot v)
&=& \sum\limits_{K^c \in \overline{J^c}} \sum\limits_{v \in S_{K^c}}
                                                \chi(u \cdot v)\nonumber\\
&=& \sum\limits_{K^c \in \overline{J^c}, I_M \cap K^c = \emptyset } (-1)^{|I \cap K^c|}(q-1)^{|M(K^c)|-|I \cap K^c|}q^{|(K^c)_M|}. \nonumber
\end{eqnarray}
From $
|M(K^c)|=|M(J^c)|\mbox{ and }|(K^c)_M|=|(J^c)_M|$ for $K^c \in \overline{J^c}$, we obtain $(i)$. In the same way, we can obtain $(ii)$.
\end{proof}

We prove that $P_E$ is uniquely determined by $Q_{E^*}$ and vice
versa.
\begin{proposition}\label{relation of p and q}
Let $\mathcal{P}$ be a poset on $[n]$, $E$ an equivalence relation
on $\mathcal{I}(\mathcal{P})$ and $E^*$ the dual relation of $E$. If
$E$ is a MacWilliams-type equivalence relation on
$\mathcal{I}(\mathcal{P})$, then
\begin{eqnarray*}
\frac{|\overline{I}|}{(q-1)^{|M(J^c)|}q^{|(J^c)_M|}}{p_{\overline{J^c}, \overline{I}}} = \frac{|\overline{J^c}|}{(q-1)^{|M(I)|}q^{|I_M|}} {q_{\overline{I},\overline{J^c}}}, \nonumber
\end{eqnarray*}
for $\overline{I} \in \mathcal{I}(\mathcal{P}) / E$ and $\overline{J^c} \in \mathcal{I}(\mathcal{P^*}) / E^*$.
\end{proposition}
\begin{proof}
Since $E$ is a MacWilliams-type equivalence relation on $\mathcal{I}(\mathcal{P})$, we have
$$
p_{_{\overline{J^c}, \overline{I}}} = \sum_{v \in S_{\overline{J^c},E^*} } \chi(u \cdot v) \mbox{ for }u \in S_{\overline{I},E}
 \mbox{ and }
q_{_{\overline{I}, \overline{J^c}}} = \sum_{u \in S_{\overline{I},E} } \chi(u \cdot v) \mbox{ for }v \in S_{\overline{J^c},E^*}.
$$
It follows from Lemmas \ref{support of u} and Proposition
\ref{calculation of P_I(J)2} that
\begin{eqnarray*}
\frac{p_{\overline{J^c}, \overline{I}}}{(q-1)^{|M(J^c)|}q^{|(J^c)_M|}} &=& \sum\limits_{K^c \in \overline{J^c}, {I} \cap (K^c)_M = \emptyset}
\left(\frac{-1}{q-1}\right)^{|I \cap K^c|}\nonumber\\
&=& \frac{1}{|\overline{I}|}\sum\limits_{L \in \overline{I}} \sum\limits_{K^c \in \overline{J^c}} \sum\limits_{{L} \cap (K^c)_M = \emptyset}
\left(\frac{-1}{q-1}\right)^{|I \cap K^c|}\nonumber\\
&=& \frac{|\overline{J^c}|}{|\overline{I}|}\sum\limits_{L \in \overline{I},{L} \cap (J^c)_M = \emptyset}
\left(\frac{-1}{q-1}\right)^{|I \cap K^c|}\nonumber\\
&=& \frac{|\overline{J^c}|}{|\overline{I}|} \frac{q_{\overline{I},\overline{J^c}}}{(q-1)^{|M(I)|}q^{|I_M|}}\nonumber.
\end{eqnarray*}
Multiplying ${|\overline{I}|}$ on both sides, the result follows.
\end{proof}

\begin{example}\label{example4}
Let $\mathcal{P}$ be an antichain on $[n]$ and $E_C$ an equivalence
relation on $\mathcal{I}(\mathcal{P})$ defined by the cardinality.
We see that
$$
S_{\overline{I},E_C} = \{u \in \mathbb{F}^n_q \mid
w_{\mathcal{P}}(u)= |I| \} \mbox{ and } S_{\overline{J^c},E_C^*} =
\{v \in \mathbb{F}^n_q \mid w_{\mathcal{P}^*}(v) = |J^c| \},
$$
for $\overline{I} \in \mathcal{I}(\mathcal{P}) / E_C$ and $\overline{J^c} \in \mathcal{I}(\mathcal{P^*}) / E_C^*$.
It follows from Lemma 6.17 in \cite{MS} that
\begin{eqnarray*}
p_{_{\overline{J^c}, \overline{I}}} = P_{|J^c|}(|I|;n) \mbox{ and } q_{_{\overline{I}, \overline{J^c}}} = P_{|I|}(|J^c|;n),
\end{eqnarray*}
where $P_{k}(x;n) := \sum\limits_{j=0}^{k} (-1)^j (q-1)^{k-j}
{{x}\choose{j}} {{n-x}\choose{k-j}}$, $k=0,1,\ldots,n$, is the
Krawtchouk polynomial. Therefore, we have $P_{E_C} = Q_{E^*_C}^{T}$.
It follows from Proposition \ref{relation of p and q} that
\begin{eqnarray*}
\frac{|\overline{I}|}{(q-1)^{|M(J^c)|}q^{|(J^c)_M|}}{p_{\overline{J^c}, \overline{I}}} = \frac{|\overline{J^c}|}{(q-1)^{|M(I)|}q^{|I_M|}} {q_{\overline{I},\overline{J^c}}}. \nonumber
\end{eqnarray*}
Since $|\overline{I}| = {{n}\choose{|I|}}$, $|\overline{J^c}| =
{{n}\choose{|J^c|}}$, $M(I) = I$, $M(J^c) = J^c$, and $I_M = (J^c)_M
= \emptyset$, we see that
\begin{eqnarray*}
\frac{{{n}\choose{|I|}}}{(q-1)^{|J^c|}}{p_{\overline{J^c},
\overline{I}}} = \frac{{{n}\choose{|J^c|}}}{(q-1)^{|I|}}
{q_{\overline{I},\overline{J^c}}}. \nonumber
\end{eqnarray*}
This coincides with Theorem 5.17 in \cite{MS}.
\end{example}

\section{Three sources of MacWilliams-type equivalence relations}

In this section, We provide three kinds of equivalence relations of
a MacWilliams-type, that is, equivalence relations defined by the
cardinality on the set of order ideals of a poset, the automorphism
of a poset and the order isomorphism on $\mathcal{I}(\mathcal{P})$
of a complement isomorphism poset. Moreover we classify posets
admitting such equivalence relations to be a MacWilliams-type
on $\mathcal{I}(\mathcal{P})$.\\

Let $f$ and $g$ be functions on the subsets of a finite set $X$. It
is known \cite{S1} that
\begin{align*}
f(A) = \sum\limits_{B \subseteq A}g(B)\mbox{ for }A \subseteq X
\mbox{ if and only if } g(A) = \sum\limits_{B \subseteq
A}(-1)^{|A|-|B|}f(B)\mbox{ for }A \subseteq X,
\end{align*}
which is called the M$\ddot{o}$bius inversion formula.\\

Now we are ready to state our main result of this section for
classifying posets admitting a MacWilliams-type equivalence
relation.

\begin{theorem}\label{automorphism group}
Let $\mathcal{P}$ be a poset on $[n]$ and $H$ a subgroup of
Aut$(\mathcal{P})$.\\
$(i)$ $E_H$ is a MacWilliams-type equivalence relation on
$\mathcal{I}(\mathcal{P})$.\\
$(ii)$ The following statements are equivalent.\\
\mbox{ } \ \ \ $(a)$ $\mathcal{P}$ is a hierarchical poset.\\
\mbox{ } \ \ \ $(b)$ $E_C$ is a MacWilliams-type equivalence relation on $\mathcal{I}(\mathcal{P})$.\\
\mbox{ } \ \ \ $(c)$ Two equivalence relations $E_{C}$ and
$E_{Aut(\mathcal{P})}$
are the same.\\
$(iii)$ The following statements are equivalent.\\
\mbox{ } \ \ \ $(a)$ $\mathcal{P}$ is a complement isomorphism poset.\\
\mbox{ } \ \ \ $(b)$ $E_S$ is a MacWilliams-type equivalence relation on
$\mathcal{I}(\mathcal{P})$.
\end{theorem}
\begin{proof} $(i)$ Note that $(I,J) \in E_H \textrm{ if and only if
}(I^c,J^c) \in E^*_H$. For $u$ and $u'$ in $S_{\overline{I},E_H}$,
let $I_1=\langle \mbox{supp}(u) \rangle_{\mathcal{P}}$ and
$I_2=\langle \mbox{supp}(u') \rangle_{\mathcal{P}}$. There is an
automorphism $\sigma$ in $H$ such that $\sigma(I_1)=I_2$. Let $J \in
\mathcal{I}(\mathcal{P})$. It follows from Proposition
\ref{calculation of P_I(J)2} that
\begin{eqnarray}
\sum\limits_{v \in S_{\overline{J^c},E_H^*}} \chi(u \cdot v) =
(q-1)^{|M(J^c)|}q^{|(J^c)_M|} \sum\limits_{K^c \in \overline{J^c},
{(I_1)}_M \cap K^c = \emptyset} \left(\frac{-1}{q-1}\right)^{|I_1
\cap K^c|}.\nonumber
\end{eqnarray}
It can be easily checked that for $A,B\subseteq\mathcal{P}$, we
obtain $\sigma(A \cap B) = \sigma(A) \cap \sigma(B)$ for all
$\sigma\in \textrm{Aut}(\mathcal{P})$ and
$\textrm{Aut}(\mathcal{P}^*) = \textrm{Aut}(\mathcal{P})$. It then
follows that
\begin{eqnarray*}
\sum\limits_{v \in S_{\overline{J^c},E_H^*}} \chi(u \cdot v)  &=& (q-1)^{|M(J^c)|}q^{|J_M^c|}
\sum\limits_{\sigma(K^c) \in \overline{J^c}, \sigma({(I_1)}_M \cap K^c)
= \emptyset}  \left(\frac{-1}{q-1}\right)^{|\sigma(I_1 \cap K^c)|}\nonumber\\
&=& (q-1)^{|M(J^c)|}q^{|J_M^c|} \sum\limits_{\sigma(K^c) \in \overline{J^c}, {(I_2)}_M \cap \sigma(K^c)
= \emptyset} \left(\frac{-1}{q-1}\right)^{|I_2 \cap \sigma(K^c)|}\nonumber\\
&=& \sum\limits_{v \in S_{\overline{J^c},E_H^*}} \chi(u' \cdot
v).\nonumber
\end{eqnarray*}
This proves Theorem \ref{S.M.I.2} $(ii)$ $(a)$. Note that $(I^c,J^c) \in
E^*_H \textrm{ if and only if }\sigma(I^c)=J^c$ for some $\sigma \in
H$. Theorem \ref{S.M.I.2} $(ii)$ $(b)$ is proved in the same argument as
above. This proves part $(i)$\\
$(ii)$ $(a) \Rightarrow (b)$ Since $\mathcal{P}$ is a hierarchical
poset, there is an automorphism $\sigma$ in $Aut(\mathcal{P})$
satisfying $\sigma(\langle \mbox{supp}(u) \rangle_{\mathcal{P}})
=\langle \mbox{supp}(u') \rangle_{\mathcal{P}}$ for $u$ and $u' \in
S_{\overline{I},E_C}$. The result is proved by Theorem \ref{S.M.I.2}
as in the proof of Theorem \ref{automorphism group}. $(a)
\Rightarrow (c)$ Let $\mathcal{P}$ be a hierarchical poset. From
the structure of $\mathcal{P}$, it is easily shown that for $I$ and
$J \in \mathcal{I}(\mathcal{P})$, $|I|=|J|$ if and only if there is
an element $\sigma$ in $Aut(\mathcal{P})$ satisfying $\sigma(I)=J$.
Hence $(a)$ implies $(c)$. $(b) \Rightarrow (a)$ It follows from
Theorem 2.5 in \cite{KO}. $(c) \Rightarrow (b)$ It follows from
Theorem \ref{automorphism group}.\\
$(iii)$ $(a) \Rightarrow (b)$ Suppose $\mathcal{P}$ is a complement
isomorphism poset. For $u,u' \in S_{\overline{I},E_S}$, there exists
an order isomorphism satisfying $\sigma(\langle \mbox{supp}(u)
\rangle_{\mathcal{P}})=\langle \mbox{supp}(u')
\rangle_{\mathcal{P}}$. Put $I_1 = \langle supp(u)
\rangle_\mathcal{P}$ and $I_2 = \langle supp(u')
\rangle_\mathcal{P}$. It follows from Lemma \ref{support of u} and
Proposition \ref{calculation of P_I(J)2} that for $\overline{J^c}
\in \mathcal{I}(\mathcal{P^*}) / E_S^*$,
\begin{eqnarray}
\sum\limits_{v \in S_{\overline{J^c},E_S^*} } \chi(u \cdot v) &=&
(q-1)^{|M(J^c)|}q^{|(J^c)_M|} \sum\limits_{K^c \in \bar{J^c}, {I_1}
\cap K^c \subseteq M(I_1)} \left(\frac{-1}{q-1}\right)^{|I_1 \cap
K^c|}.\nonumber
\end{eqnarray}
Replacing $I_1 \cap K^c$ by $A$, we have
\begin{eqnarray}
\sum\limits_{v \in S_{\overline{J^c},E_S^*} } \chi(u \cdot v) &=&
(q-1)^{|M(J^c)|}q^{|(J^c)_M|} \sum\limits_{A \subseteq M(I_1)}
\left(\frac{-1}{q-1}\right)^{|A|} \sum\limits_{K^c \in
\overline{J^c}, I_1 \cap K^c = A} 1. \nonumber
\end{eqnarray}
Applying the M$\ddot{o}$bius inversion formula, we obtain
\begin{eqnarray}
\frac{\sum\limits_{v \in S_{\overline{J^c},E_S^*} } \chi(u \cdot
v)}{(q-1)^{|M(J^c)|}q^{|(J^c)_M|}} = \sum\limits_{A \subseteq
M(I_1)} \left(\frac{-1}{q-1}\right)^{|A|} \sum\limits_{B \subseteq
A} (-1)^{|A \setminus B|}\sum\limits_{K^c \in \overline{J^c}, I_1
\cap K^c \subseteq B} 1. \nonumber
\end{eqnarray}
Let $B \subseteq A \subseteq M(I_1)$. One can easily check that
$|A|=|\sigma(A)|$, $|A \setminus B|=|\sigma(A) \setminus
\sigma(B)|$, and $(I_1 \setminus B,I_2 \setminus \sigma(B)) \in E_S$
since $\sigma : I_1 \rightarrow I_2$ is an order isomorphism. Since
$\mathcal{P}$ is a complement isomorphism poset, $(I_1\setminus B)^c
\simeq (I_2\setminus \sigma(B))^c$. Hence we have
$$
\sum\limits_{K^c \in \overline{J^c}, I_1 \cap K^c \subseteq B} 1 =
\sum\limits_{K^c \in \overline{J^c}, K^c \subseteq (I_1 \setminus
B)^c} 1 = \sum\limits_{K^c \in \overline{J^c}, K^c \subseteq (I_2
\setminus \sigma(B))^c} 1 = \sum\limits_{K^c \in \overline{J^c}, I_2
\cap K^c \subseteq \sigma(B)} 1.
$$
It follows that
\begin{eqnarray}
& &\frac{1}{(q-1)^{|M(J^c)|}q^{|(J^c)_M|}}\sum\limits_{v \in S_{\overline{J^c},E_S^*}} \chi(u \cdot v) \nonumber \\
&=&  \sum\limits_{A \subseteq M(I_1)}
\left(\frac{-1}{q-1}\right)^{|A|}
\sum\limits_{B \subseteq A} (-1)^{|A \setminus B|}\sum\limits_{K^c \in \overline{J^c}, I_1 \cap K^c \subseteq B} 1 \nonumber\\
&=&  \sum\limits_{\sigma(A) \subseteq M(I_2)}
\left(\frac{-1}{q-1}\right)^{|\sigma(A)|}
  \sum\limits_{\sigma(B) \subseteq \sigma(A)} (-1)^{|\sigma(A) \setminus \sigma(B)|}
 \sum\limits_{K^c \in \overline{J^c}, I_2 \cap K^c \subseteq \sigma(B)} 1\nonumber\\
&=& \frac{1}{(q-1)^{|M(J^c)|}q^{|(J^c)_M|}}\sum\limits_{v \in
S_{\overline{J^c},E_S^*}} \chi(u' \cdot v).\nonumber
\end{eqnarray}
This proves Theorem \ref{S.M.I.2} $(ii)$ $(a)$. Since Theorem \ref{S.M.I.2}
$(ii)$ $(b)$ can be proved in the same way, the result follows.\\
$(b) \Rightarrow (a)$ Suppose $\mathcal{P}$ is not a complement
isomorphism poset. Then there are $I_1$ and $I_2$ on
$\mathcal{I}(\mathcal{P})$ such that $I_1 \simeq I_2$ and $I_1^c
\not\simeq I_2^c$. Let $\mathcal{C}_1$ and $\mathcal{C}_2$ be linear
codes of $\mathbb{F}_q^{n}$ such that
\begin{eqnarray*}
\mathcal{C}_i = \{x \in \mathbb{F}_q^{n}  \vert \mbox{ supp}(x) \subseteq I_{i}\}, \quad i=1,2.
\end{eqnarray*}
It follows that $W(\mathcal{C}_1,\mathcal{P},E_S)=W(\mathcal{C}_2,\mathcal{P},E_S)$.
The dual codes $\mathcal{C}_1^{\perp}$ and $\mathcal{C}_2^{\perp}$ are given by
\begin{eqnarray*}
\mathcal{C}^{\perp}_i = \{x \in \mathbb{F}_q^{n}  \vert \mbox{ supp}(x) \subseteq I_{i}^c\}, \quad i=1,2.
\end{eqnarray*}
From Lemma \ref{sphere}, we have $A_{\overline{I^c_1},
E^*_S}(\mathcal{C}_1^{\perp}) = (q-1)^{|M(I^c_1)|}q^{|(I^c_1)_M|}$.
Note that $|I^c|=|J^c|$ because $I^c \simeq J^c$. If $x \in
\mathcal{C}^{\perp}_2 $ such that $\langle \mbox{ supp}(x)
\rangle_{\mathcal{P}^*} \simeq I^c_1$, then $|\langle \mbox{
supp}(x) \rangle_{\mathcal{P}^*}|=|I^c_1|=|I^c_2|$. It then follows
from $\langle \mbox{ supp}(x) \rangle_{\mathcal{P}^*}\subseteq
I^c_2$ that $I^c_2=\langle \mbox{ supp}(x)
\rangle_{\mathcal{P}^*}\simeq I^c_1$, which is a contradiction to
the fact that $I_1^c \not\simeq I_2^c$. It follows that
$W(\mathcal{C}_1^{\perp}, \mathcal{P^*},E^*_S) \neq
W(\mathcal{C}_2^{\perp}, \mathcal{P^*}, E^*_S)$. Therefore, we have
the result.
\end{proof}

\begin{corollary}
Let $\mathcal{P}$ be a poset on $[n]$ and $H$ a subgroup of
Aut$(\mathcal{P})$. For $\overline{I} \in \mathcal{I}(\mathcal{P}) /
E_H$ and $\overline{J^c} \in \mathcal{I}(\mathcal{P^*}) / E_H^*$, we
have
\begin{eqnarray*}
\frac{|\{\sigma \in H \mid \sigma(J^c) =
J^c\}|}{(q-1)^{|M(J^c)|}q^{|(J^c)_M|}}{p_{\overline{J^c},
\overline{I}}} = \frac{|\{\sigma \in H \mid \sigma(I) =
I\}|}{(q-1)^{|M(I)|}q^{|I_M|}} {q_{\overline{I},\overline{J^c}}}.
\nonumber
\end{eqnarray*}
\end{corollary}

\begin{proof}
From Theorem \ref{automorphism group}, the equivalence relation
$E_H$ on $\mathcal{I}(\mathcal{P})$ is a MacWilliams-type. It
follows from Proposition \ref{relation of p and q} that
\begin{eqnarray*}
\frac{|\overline{I}|}{(q-1)^{|M(J^c)|}q^{|(J^c)_M|}}{p_{\overline{J^c},
\overline{I}}} = \frac{|\overline{J^c}|}{(q-1)^{|M(I)|}q^{|I_M|}}
{q_{\overline{I},\overline{J^c}}}, \nonumber
\end{eqnarray*}
for $\overline{I} \in \mathcal{I}(\mathcal{P}) / E_H$ and
$\overline{J^c} \in \mathcal{I}(\mathcal{P^*}) / E_H^*$. Since $H =
|\overline{I}| |\{\sigma \in H \mid \sigma(I) = I\}| =
|\overline{J^c}| |\{\sigma \in H \mid \sigma(J^c) = J^c\}|$, the
result follows.
\end{proof}





\bibliographystyle{model1b-num-names}
\bibliography{<your-bib-database>}







\end{document}